\documentclass[reqno]{amsproc}
\usepackage{amssymb,tikz}
\usepackage{graphicx}
\usepackage{euscript}
\usepackage{mathrsfs} 
\usepackage{units}   
\usepackage{color}

\makeatletter
\@namedef{subjclassname@2010}{%
\textup{2010} Mathematics Subject Classification}
\makeatother

\numberwithin{equation}{section}

\newtheorem{thm}{Theorem}[section]
\newtheorem{cor}[thm]{Corollary}
\newtheorem{lem}[thm]{Lemma}

\newtheorem*{thm*}{Theorem}

\theoremstyle{remark}
\newtheorem{rem}[thm]{Remark}

\newtheorem{opq}[thm]{Problem}

\theoremstyle{definition}

\newtheorem*{dfn*}{Definition}

\DeclareMathOperator{\D}{d}
\DeclareMathOperator{\dess}{{\mathsf{Des}}}
\DeclareMathOperator{\dzii}{{\mathsf{Chi}}}

\DeclareMathOperator{\koo}{{\mathsf{root}}}

\DeclareMathOperator{\paa}{{\mathsf{par}}}

\newcommand*{\ascr}{\mathscr A}
\newcommand*{\at}[1]{\mathsf{At}(#1)}
\newcommand*{\borel}[1]{{\mathfrak B}(#1)}
\newcommand*{\cbb}{\mathbb C}

\newcommand*{\co}{\EuScript{CO}}
\newcommand*{\des}[1]{{\dess(#1)}}
\newcommand*{\deso}[1]{{\dess(#1)^{\circ}}}
\newcommand*{\dz}[1]{{\EuScript D}(#1)}
\newcommand*{\dzi}[1]{\dzii(#1)}
\newcommand*{\dzin}[2]{\dzii^{\langle#1\rangle}(#2)}

\newcommand*{\Ge}{\geqslant}
\newcommand*{\hh}{\mathcal H}

\newcommand*{\kk}{\mathcal K}

\newcommand*{\lambdab}{{\boldsymbol\lambda}}

\newcommand*{\Le}{\leqslant}

\newcommand*{\nbb}{\mathbb N}

\newcommand*{\pa}[1]{\paa(#1)}

\newcommand*{\pscr}[1]{\EuScript{P}_{#1}(\rbb_+)}

\newcommand*{\rbb}{\mathbb R}

\newcommand*{\slam}{S_{\boldsymbol \lambda}}
\newcommand*{\smalloplus}{\raise0pt\hbox{$\scriptscriptstyle \oplus$}}

\newcommand*{\tcal}{{\mathscr T}}

\newcommand*{\ws}{\EuScript{WS}}
\newcommand*{\xbf}{\boldsymbol{x}}
\newcommand*{\xsc}{\mathscr X}

\newcommand*{\zbb}{\mathbb Z}

\begin{document}
   \title[Subnormal weighted shifts whose $n$th powers have trivial domain]{Subnormal weighted shifts on directed trees \\ whose $n$th powers have trivial domain}
   \author[P.\ Budzy\'{n}ski]{Piotr Budzy\'{n}ski}
   \address{Katedra Zastosowa\'{n} Matematyki,
Uniwersytet Rolniczy w Krakowie, ul.\ Balicka 253c,
PL-30198 Krak\'ow, Poland}
\email{piotr.budzynski@ur.krakow.pl}
   \author[Z.\ J.\ Jab{\l}o\'nski]{Zenon Jan
Jab{\l}o\'nski}
   \address{Instytut Matematyki,
Uniwersytet Jagiello\'nski, ul.\ \L ojasiewicza 6,
PL-30348 Kra\-k\'ow, Poland}
\email{Zenon.Jablonski@im.uj.edu.pl}
   \author[I.\ B.\ Jung]{Il Bong Jung}
   \address{Department of Mathematics,
Kyungpook National University, Daegu 702-701,
Korea}
   \email{ibjung@knu.ac.kr}
   \author[J.\ Stochel]{Jan Stochel}
\address{Instytut Matematyki, Uniwersytet
Jagiello\'nski, ul.\ \L ojasiewicza 6, PL-30348
Kra\-k\'ow, Poland} \email{Jan.Stochel@im.uj.edu.pl}
   \thanks{The research of the first author was
supported by the NCN (National Science Center), decision No. DEC-2011/01/D/ST1/05805. The research of the second and fourth authors was supported by the NCN (National Science Center), decision No. DEC-2013/11/B/ST1/03613. The research of the third author was supported by the National Research Foundation of Korea (NRF) grant funded by the Korea government (MSIP) (No. 2009-0083521).}
   \begin{abstract}
It is shown that for any positive integer $n$ there
exists a subnormal weighted shift on a directed tree
whose $n$th power is closed and densely defined while its $(n+1)$th power has trivial domain. Similar result for composition operators in $L^2$-spaces is established.
   \end{abstract}
   \maketitle
   \section{Introduction}
In 1940 Naimark gave a remarkable example of a closed
symmetric operator whose square has trivial domain
(cf.\ \cite{nai}). In 1983 Chernoff published a short example of a semibounded closed symmetric operator whose square has trivial domain (cf.\ \cite{cher}). In the same year  Schm\"{u}dgen found out another pathological behaviour of domains of powers of closed symmetric operators related to density with respect to graph norms (cf.\ \cite{Schm}). It turns out that Naimark's phenomenon can never happen in some concrete classes of operators. Among them are the class $\co$ of composition operators in $L^2$-spaces and the class $\ws$ of weighted shifts on directed trees. The reason for this is that symmetric operators in these classes are automatically bounded (see \cite{j-j-s3,b-j-j-sS}).

The class $\co$ has been attracting attention of a
considerable number of researchers since at least late
1950's. We refer the reader to \cite{sin-man} and
\cite{b-j-j-sC} for more information on bounded and
unbounded operators in the class $\co$, respectively.
The class $\ws$ was introduced in \cite{j-j-s} and has
been intensively studied since then (see e.g.,
\cite{b-j-j-sA,j-j-s2,b-j-j-sB,j-j-s3}). It
substantially generalizes the class of (unilateral and
bilateral) weighted shifts in $\ell^2$-spaces. It is
also related to the class of operators investigated by
Carlson in \cite{Ca1,Ca2}. Unbounded weighted shifts
on directed trees proved to have very interesting
features which make them desirable candidates for
testing hypothesises and constructing examples (see
e.g., \cite{j-j-s,j-j-s2,j-j-s4,b-d-j-s,Trep}). This
is due to the fact that the interplay between graph
theory and operator theory makes the class $\ws$ more
flexible.

The above raises the question of whether the square,
or a higher power, of an operator in the class $\ws$
or $\co$ has trivial domain. Clearly, such an operator
must be nonsymmetric. The question becomes interesting
and highly nontrivial when the operator under
consideration is assumed to be subnormal (recall that
symmetric operators are subnormal, cf.\ \cite[Theorem
1 in Appendix I.2]{a-g}). One of the reasons for this
is that quasinormal operators which are particular
instances of subnormal operators have all powers
densely defined (cf.\ \cite[Proposition 5]{StSz1}). On
the other hand, formally normal
operators\footnote{Formally normal operators are
natural generalizations of symmetric operators. In
general, they are not subnormal (cf.\ \cite{Codi}).}
belonging to the class $\ws$ or $\co$ are
automatically normal (cf.\ \cite[Proposition
3.1]{j-j-s3} and \cite[Theorem 9.4]{b-j-j-sC}), and as
such have all powers densely defined. Some attempts to
tackle our question have been undertaken in
\cite{j-j-s4,Bu} where the case of hyponormal
operators in both classes $\ws$ and $\co$ was solved.
Recently, it has been shown that for every positive
integer $n$ there exists an injective subnormal
operator in the class $\ws$ whose $n$th power is
densely defined while its $(n+1)$th power is not; the
same is true for $\co$ (cf.\ \cite{b-d-j-s}). These
examples are built over the simplest possible directed
trees which admit such operators.

In view of the above discussion, the following
problem arises (the case of $n=1$ appeared already in
\cite{j-j-s4}):
   \begin{opq}  \label{Q}
Is it true that for every integer $n \Ge 1$, there
exists a subnormal weighted shift on a
directed tree whose $n$th power is densely defined and the domain of its $(n+1)$th power is trivial?
   \end{opq}
In the present paper we solve Problem \ref{Q} affirmatively (cf.\ Theorem \ref{main}).  A similar problem can be stated for composition operators in $L^2$-spaces. We solve it  affirmatively as well (cf.\ Corollary \ref{sq5}).
   \section{Preliminaries}
First, we introduce some notation and terminology. In
what follows $\zbb_+$, $\nbb$, $\rbb_+$ and $\cbb$
stand for the sets of nonnegative integers, positive
integers, nonnegative real numbers and complex
numbers, respectively. For $n\in \nbb$, we denote by
$\nbb^n$ the $n$-fold Cartesian product of $\nbb$ with
itself. We set $J_n = \{k \in \nbb\colon k \Le n\}$
for $n\in \nbb$. We write $\borel{\rbb_+}$ for the
$\sigma$-algebra of all Borel subsets of $\rbb_+$.
Given $\vartheta \in \rbb_+$, we denote by
$\pscr{\vartheta}$ the set of all Borel probability
measures on $\rbb_+$ whose closed
supports\footnote{\;Recall that a finite Borel measure
on $\rbb_+$ is regular and as such has a closed
support.} are contained in $[\vartheta,\infty)$, and
by $\delta_{\vartheta}$ the measure in
$\pscr{\vartheta}$ concentrated on the one-point set
$\{\vartheta\}$ (all measures considered in this paper
are positive). The notation $\bigsqcup$ is reserved to
denote pairwise disjoint union of sets.

The following auxiliary lemma concerning moments is
stated without proof. Here and later,
$\int_a^{\infty}$ means integration over the closed
interval $[a, \infty)$ on the real~ line.
   \begin{lem} \label{zespol}
Suppose $\mu$ is a finite Borel measure on $\rbb_+$
such that $\int_0^\infty s^{n} \D \mu(s) < \infty$ for
some $n \in \nbb$. Then $\int_0^\infty s^{k} \D \mu(s)
< \infty$ for every $k \in \nbb$ such that $k \Le n$.
   \end{lem}
The domain of an operator $A$ in a complex Hilbert
space $\hh$ is denoted by $\dz{A}$ (all operators
considered in this paper are linear). Recall that a
closed densely defined operator $A$ in $\hh$ is said
to be {\em normal} if $AA^* = A^*A$, where $A^*$ stands for the adjoint of $A$ (see
\cite{b-s,Schm2,Weid} for more on this class of
operators). We say that a densely defined operator $A$
in $\hh$ is {\em subnormal} if there exists a complex
Hilbert space $\kk$ and a normal operator $N$ in $\kk$
such that $\hh \subseteq \kk$ (isometric embedding)
and $Ah = Nh$ for all $h \in \dz{S}$. We refer the
reader to \cite{Con} and
\cite{StSz3,StSz1,StSz4,StSz2} for the foundations of
the theory of bounded and unbounded subnormal
operators, respectively.

Let $\tcal=(V,E)$ be a directed tree, where $V$ and
$E$ stand for the sets of {\em vertices} and {\em
edges} of $\tcal$, respectively. Set $\dzi u = \{v\in
V\colon (u,v)\in E\}$ for $u \in V$. Denote by $\paa$
the partial function from $V$ to $V$ which assigns to
a vertex $u\in V$ its parent $\pa{u}$ (i.e.\ a unique
$v \in V$ such that $(v,u)\in E$). A vertex $u \in V$
which has no parent is called a {\em root} of $\tcal$;
if it exists, it is unique and denoted by $\koo$. Set
$V^\circ=V \setminus \{\koo\}$ if $\tcal$ has a root;
otherwise, we put $V^\circ=V$. If $W \subseteq V$, we
set $\dzi{W }= \bigcup_{v \in W} \dzi v$, $\dzin{0}{W}
= W$ and $\dzin{n+1}{W} = \dzi{\dzin{n}{W}}$ for every
$n\in \zbb_+$. Given $u\in V$, we put $\dzin n{u} =
\dzin n{\{u\}}$ and $\des{u}= \bigcup_{n=0}^\infty
\dzin n{u}$. Since $(\des{u}, E \cap (\des{u}\times
\des{u}))$ is a subtree of $\tcal$, we see that
$\deso{u}=\des{u} \setminus \{u\}$ for all $u \in V$.
We say that $\tcal$ is {\em extremal} if $\dzi{u}$ is
countably infinite for every $u \in V$. It is easily
seen that up to isomorphism of graphs, there are
exactly two extremal directed trees, one with root,
the other without.

Denote by $\ell^2(V)$ the Hilbert space of square
summable complex functions on $V$ with standard
inner product. Given $u \in V$, we write $e_u$
for the characteristic function of the one-point
set $\{u\}$. Clearly, the system $\{e_u\}_{u\in
V}$ is an orthonormal basis of $\ell^2(V)$. For
$\lambdab=\{\lambda_v\}_{v \in V^{\circ}}
\subseteq \cbb$, the operator $\slam$ in
$\ell^2(V)$ defined by
   \begin{align*} \begin{aligned}
\dz {\slam} & = \{f \in \ell^2(V) \colon
\varLambda_\tcal f \in \ell^2(V)\},
   \\
\slam f & = \varLambda_\tcal f, \quad f \in
\dz{\slam},
\end{aligned}
\end{align*}
where $\varLambda_\tcal$ is the mapping defined
on functions $f\colon V \to \cbb$ via
   \begin{align*}
(\varLambda_\tcal f) (v) =
   \begin{cases}
\lambda_v \cdot f\big(\pa v\big) & \text{ if }
v\in V^\circ,
   \\
0 & \text{ if } v=\koo,
   \end{cases}
   \end{align*}
is called a {\em weighted shift on} $\tcal$ with
weights $\lambdab$. Recall that unilateral or
bilateral weighted shifts are instances of weighted
shifts on directed trees. We refer the reader to
\cite{j-j-s} for basic facts about directed trees and
weighted shifts on directed trees needed in this
paper.

Below we state a criterion for subnormality of
weighted shifts on countably infinite directed trees.
It is an extension, in a sense, of \cite[Theorem
5.1.1]{b-j-j-sA} to the case of weighted shifts on
directed trees whose $C^\infty$ vectors are not
necessarily dense in the underlying space. This
criterion helps us to solve Problem \ref{Q}.
   \begin{thm}[\mbox{\cite[Theorem 3]{b-d-j-s}}]
   \label{wsi} Let $\slam$ be a weighted shift on a
countably infinite directed tree $\tcal=(V,E)$ with
weights $\lambdab=\{\lambda_v\}_{v \in V^{\circ}}$.
Suppose there exist a family $\{\mu_v\}_{v \in V}$ of
Borel probability measures on $\rbb_+$ and a family
$\{\varepsilon_v\}_{v\in V}$ of nonnegative real
numbers such that\/\footnote{\;We adopt the
conventions that $0\cdot \infty = \infty \cdot 0 = 0$,
$\frac{1}{0} = \infty$ and $\sum_{v\in \varnothing}
\xi_v = 0$.}
   \begin{align} \label{consist6}
\mu_u(\varDelta) = \sum_{v \in \dzi{u}} |\lambda_v|^2
\int_\varDelta \frac{1}{s} \D \mu_v(s) + \varepsilon_u
\delta_0(\varDelta), \quad \varDelta \in
\borel{\rbb_+}, \, u \in V.
   \end{align}
Then the following two assertions hold\/{\em :}
\begin{enumerate}
\item[(i)] if $\slam$ is densely defined, then
$\slam$ is subnormal,
\item[(ii)] if $n\in \nbb$, then $\slam^n$
is densely defined if and only if $\int_0^\infty s^n
\, \D \mu_u(s) < \infty$ for every $u \in V$ such that
$\dzi{u}$ has at least two vertices.
   \end{enumerate}
   \end{thm}
   \begin{rem}
Note that if $w \in V^{\circ}$, $\lambda_w \neq 0$ and
equality in \eqref{consist6} holds for $u\in \{w,
\pa{w}\}$, then $\varepsilon_w = 0$. Indeed,
substituting $\varDelta=\{0\}$ into \eqref{consist6}
with $u=\pa{w}$, we deduce that $\mu_w(\{0\})=0$. As a
consequence, we see that \eqref{consist6} yields
$\mu_v(\{0\})=0$ for every $v \in V^{\circ}$ such that
$\lambda_v \neq 0$. Hence, applying the same procedure
to $u=w$ gives $\varepsilon_w = 0$. This implies that
if all the weights $\{\lambda_v \colon v \in
V^{\circ}\}$ are nonzero, then condition
\eqref{consist6} takes the following simplified form
   \begin{align*}
\mu_u(\varDelta) =
   \begin{cases}
\displaystyle{ \sum_{v \in \dzi{u}} |\lambda_v|^2
\int_\varDelta \frac{1}{s} \D \mu_v(s)} & \text{ if }
u \in V^{\circ},
   \\[3ex]
\displaystyle{\sum_{v \in \dzi{\koo}} |\lambda_v|^2
\int_\varDelta \frac{1}{s} \D \mu_v(s) +
\varepsilon_{\koo} \delta_0(\varDelta)} & \text{ if }
u = \koo,
   \end{cases}
\quad \varDelta \in \borel{\rbb_+}.
   \end{align*}
   \end{rem}
   The following lemma will be used in the proof of
the main theorem.
   \begin{lem} \label{ddn}
Let $\slam$ be a weighted shift on a directed
tree $\tcal=(V,E)$ with weights
$\lambdab=\{\lambda_v\}_{v \in V^{\circ}}$ and
let $n\in \nbb$. Then the following two
conditions are equivalent{\em :}
   \begin{enumerate}
   \item[(i)] $\dz{\slam^n} = \{0\}$,
   \item[(ii)] $e_u\notin \dz{\slam^n}$ for every $u\in
V$.
   \end{enumerate}
Moreover, if there exist a family $\{\mu_v\}_{v \in
V}$ of Borel probability measures on $\rbb_+$ and a
family $\{\varepsilon_v\}_{v\in V} \subseteq \rbb_+$
which satisfy \eqref{consist6}, then {\em (i)} is
equivalent to
   \begin{enumerate}
   \item[(iii)] $\int_0^\infty s^n \D \mu_u(s) = \infty$
for every $u \in V$.
   \end{enumerate}
   \end{lem}
   \begin{proof}
(i)$\Rightarrow$(ii) Evident.

(ii)$\Rightarrow$(i) Suppose that, contrary to
our claim, there exists $f\in \dz{\slam^n}$ such
that $f \neq 0$. Then $f(u) \neq 0$ for some $u
\in V$. In view of \cite[Theorem
3.2.2(ii)]{j-j-s2}, this implies that $e_u \in
\dz{\slam^n}$, a contradiction.

(ii)$\Leftrightarrow$(iii) Apply Lemma
\ref{zespol} and \cite[Lemmata 2.3.1(i) and
4.2.2(i)]{b-j-j-sA}.
   \end{proof}
   \section{The main theorem}
We begin by recalling that if there exists a weighted
shift $\slam$ on a directed tree $\tcal$ with nonzero
weights such that
   \begin{align} \label{cykada}
\text{$\slam$ is densely defined and $\dz{\slam^2}=\{0\}$,}
   \end{align}
then the directed tree $\tcal$ is extremal (cf.\
\cite[Theorem 3.1]{j-j-s4}). As shown in \cite[Theorem
3.1]{j-j-s4}, each extremal directed tree admits a
hyponormal weighted shift $\slam$ with nonzero weights
that satisfies \eqref{cykada}. Hence, to solve Problem
\ref{Q} affirmatively we may assume that the directed
tree in question is extremal.

   The following theorem is the main result of the
present paper. It solves Problem~ \ref{Q}
affirmatively. The proof of Theorem \ref{main} is
given in Section \ref{proofm}.
   \begin{thm} \label{main}
Suppose $\tcal=(V,E)$ is an extremal directed tree and
$n\in \nbb$. Then there exists a subnormal weighted
shift $\slam$ on $\tcal$ with nonzero weights such
that $\slam^n$ is densely defined and
$\dz{\slam^{n+1}}=\{0\}$.
   \end{thm}
   The following simple observation which is related
to Problem \ref{Q} is stated without proof.
   \begin{lem} \label{nowl}
If $A$ is an operator such that $\dz{A^n}=\{0\}$ for
some positive integer $n$, then $A$ is injective.
   \end{lem}
By Lemma \ref{nowl}, the operator $\slam$ in Theorem
\ref{main} is automatically injective.
   \begin{rem}
It is worth mentioning that every weighted shift on a
directed tree is closed (cf.\ \cite[Proposition
3.1.2]{j-j-s}). However, it is not true that powers of
weighted shifts on directed trees are closed. In fact,
it may happen that the square of an unbounded
injective unilateral shift $S$ is bounded and
consequently $S^2$ is not closed (see e.g.,\ \cite[p.\
198]{St}). On the other hand, if a subnormal operator
is closed, then all its powers are closed (cf.\
\cite[Proposition 6]{StSz5}; see also
\cite[Proposition 5.3]{St}). In particular, all powers
of the operator $\slam$ in Theorem \ref{main} are
closed.
   \end{rem}
Theorem \ref{main} has a counterpart for composition
operators in $L^2$-spaces. Recall that if
$(X,\ascr,\mu)$ is a $\sigma$-finite measure space and
$\phi\colon X \to X$ is a transformation such that
$\phi^{-1}(\varDelta) \in \ascr$ for every $\varDelta
\in \ascr$, and $\mu(\phi^{-1}(\varDelta)) = 0$ for
every $\varDelta \in \ascr$ such that $\mu(\varDelta)
= 0$, then the operator $C\colon L^2(\mu) \supseteq
\dz{C} \to L^2(\mu)$ given by
   \begin{align*}
\dz{C} = \{g \in L^2(\mu)\colon g \circ \phi \in
L^2(\mu)\} \text{ and } C f = f \circ \phi \text{ for
} f \in \dz{C}
   \end{align*}
is well-defined; we call it a {\em composition}
operator. Composition operators are always closed (see
e.g., \cite[Proposition 3.2]{b-j-j-sC}), but in
general their powers are not (cf.\ \cite[Example
5.4]{b-j-j-sC}). However, if the composition operator
is subnormal, then all its powers are closed.
   \begin{cor} \label{sq5}
For every $n\in \nbb$, there exists a subnormal
composition operator $C$ in an $L^2$-space over
$\sigma$-finite measure space such that $C^n$ is
densely defined and $\dz{C^{n+1}}=\{0\}$.
   \end{cor}
   \begin{proof}
Apply Theorem \ref{main}, \cite[Theorem 3.2.1]{j-j-s}
and \cite[Lemma 4.3.1]{j-j-s2}.
   \end{proof}
It is worth mentioning that in view of Lemma
\ref{nowl}, the operator $C$ in Corollary \ref{sq5} is
automatically injective. A close inspection of the
proofs of Theorem \ref{main} and \cite[Lemma
4.3.1]{j-j-s2} reveals that the operator $C$ in
Corollary \ref{sq5} can be built on a discrete measure
space. The continuous case can be easily derived from
the discrete one by applying \mbox{\cite[Theorem
2.7]{Jab}}.
   \section{The proof of the main theorem} \label{proofm}
Since the proof of the main theorem is quite long, we
divide it into several lemmas. To begin with, we
recall the following definition: a Borel measure $\mu$
on $\rbb_+$ is said to be {\em discrete} if there
exist a countable subset $\varDelta$ of $\rbb_+$ and a
family $\{\alpha_{t}\}_{t \in \varDelta}$ of positive
real numbers such that $\mu = \sum_{t \in \varDelta}
\alpha_{t} \delta_{t}$. The set $\varDelta$, which is
uniquely determined by $\mu$, is denoted by $\at{\mu}$
(if $\varDelta=\varnothing$, then $\mu=0$).

For the reader's convenience, we include the proof of
the following result which seems to be folklore (the
idea of the proof comes from \cite[Example
1]{b-d-j-s}).
   \begin{lem}
If $m \in \nbb$ and $\varDelta$ is a countable subset
of $\rbb_+$ such that $\sup \varDelta = \infty$, then
there exists a finite discrete Borel measure $\mu$ on
$\rbb_+$ such that $\at{\mu}=\varDelta$,
$\int_0^\infty s^m \D \mu(s) < \infty$ and
$\int_0^\infty s^{m+1} \D \mu(s) = \infty$.
   \end{lem}
   \begin{proof}
By our assumptions $\varDelta$ is countably infinite.
Hence there exists a sequence $\{t_j\}_{j=1}^\infty$
of distinct real numbers such that $\varDelta =
\{t_j\colon j \in \nbb\}$. Since $\sup_{j\in \nbb}t_j
= \infty$, there exists a subsequence
$\{t_{j_k}\}_{k=1}^\infty$ of the sequence
$\{t_j\}_{j=1}^\infty$ such that $t_{j_k} \Ge k$ for
every $k \in \nbb$. Set $\varOmega = \{j_k\colon k \in
\nbb\}$. Clearly, there exists a family
$\{\beta_j\}_{j \in \nbb \setminus \varOmega}$ of
positive real numbers such that
   \begin{align} \label{momenty1}
\sum_{j \in \nbb \setminus \varOmega} \beta_j t_j^m <
\infty.
   \end{align}
Define the family $\{\beta_j\}_{j\in \varOmega}$ of
positive real numbers by
   \begin{align*}
\beta_{j_k} = \frac{1}{k^{2}t_{j_k}^m}, \quad k \in
\nbb.
   \end{align*}
Since $t_{j_k} \Ge k$ for every $k\in \nbb$, we
have
   \begin{align} \label{slon6}
\sum_{j \in \varOmega} \beta_j t_j^m = \sum_{k =
1}^\infty \beta_{j_k} t_{j_k}^m = \sum_{k = 1}^\infty
\frac{1}{k^2} < \infty
   \end{align}
and
   \begin{align}  \label{slon7}
\sum_{j \in \varOmega} \beta_j t_j^{m+1} = \sum_{k =
1}^\infty \beta_{j_k} t_{j_k}^{m+1} = \sum_{k =
1}^\infty \frac{t_{j_k}}{k^2} \Ge \sum_{k = 1}^\infty
\frac{1}{k} = \infty.
   \end{align}
Combining \eqref{momenty1}, \eqref{slon6} and
\eqref{slon7}, we deduce that the measure $\mu :=
\sum_{t \in \varDelta} \alpha_{t} \delta_{t}$ with
$\alpha_{t_j} = \beta_j$ for $j\in \nbb$ meets our
requirements. This completes the proof.
   \end{proof}
   \begin{cor} \label{momenty-c}
If $m \in \nbb$, $\vartheta \in \rbb_+$ and $E$
is a countably infinite subset of $\rbb_+$, then
there exists a finite discrete Borel measure $\mu$
on $\rbb_+$ such that $\at{\mu}$ is a countably
infinite subset of $[\vartheta, \infty)$, $E \cap
\at{\mu} = \varnothing$, $\int_0^\infty s^m \D
\mu(s) < \infty$ and $\int_0^\infty s^{m+1} \D
\mu(s) = \infty$.
   \end{cor}
Set $\xsc = \bigcup_{k=0}^\infty \xsc_k$, where
$\xsc_k = \bigsqcup_{j=0}^{k} \nbb^j$ with $\nbb^0 =
\{0\}$.
   \begin{lem} \label{sq1}
If $n\in \nbb$ and $\vartheta\in \rbb_+$, then there
exists a family $\{\nu_{\xbf}\}_{\xbf\in \xsc}$ of
finite discrete Borel measures on $\rbb_+$ such that
   \begin{enumerate}
   \item[(i)] $\{\at{\nu_{\xbf}}\}_{\xbf\in \xsc}$ are
pairwise disjoint countably infinite subsets of
$[\vartheta, \infty)$,
   \item[(ii)]  $\sum_{\xbf \in \nbb^k} \int_0^{\infty}
s^{k+n} \D \nu_{\xbf}(s) \Le 2^{-k}$ for all $k \in
\zbb_+$,
   \item[(iii)] $\int_0^{\infty} s^{k+n+1} \D
\nu_{\xbf}(s) = \infty$ for all $\xbf \in \nbb^k$ and
all $k\in \zbb_+$.
   \end{enumerate}
   \end{lem}
   \begin{proof}
We use an induction argument. First, by Corollary
\ref{momenty-c}, there exists a finite discrete Borel
measure $\nu_0$ on $\rbb_+$ such that $\at{\nu_{0}}$
is countably infinite subset of $[\vartheta, \infty)$,
$\int_0^{\infty} s^{n} \D \nu_{0}(s) < \infty$ and
$\int_0^{\infty} s^{n+1} \D \nu_{0}(s) = \infty$. The
induction step is as follows. Fix $k \in \zbb_+$, and
suppose we have constructed a family
$\{\nu_{\xbf}\}_{\xbf \in \xsc_k}$ of finite discrete
Borel measures on $\rbb_+$ such that
$\{\at{\nu_{\xbf}}\}_{\xbf\in \xsc_k}$ are pairwise
disjoint countably infinite subsets of $[\vartheta,
\infty)$,
   \begin{align} \label{ZenWielkiJest}
\int_0^{\infty} s^{j+n} \D \nu_{\xbf}(s) < \infty
\text{ and } \int_0^{\infty} s^{j+n+1} \D
\nu_{\xbf}(s) = \infty
   \end{align}
for all $\xbf \in \nbb^j$ and all $j \in \{0, \ldots,
k\}$. Let $\iota_k\colon \nbb \to \nbb^{k+1}$ be any
bijection. Applying Corollary \ref{momenty-c} to
$E=\bigsqcup_{\xbf \in \xsc_k} \at{\nu_{\xbf}}$, we
find a finite discrete Borel measure
$\nu_{\iota_k(1)}$ on $\rbb_+$ such that
$\at{\nu_{\iota_k(1)}}$ is a countably infinite subset
of $[\vartheta, \infty)$, $\at{\nu_{\iota_k(1)}} \cap
E = \varnothing$, $\int_0^{\infty} s^{n+k+1} \D
\nu_{\iota_k(1)}(s) < \infty$ and $\int_0^{\infty}
s^{n+k+2} \D \nu_{\iota_k(1)}(s) = \infty$. Using
induction on $i$, we obtain a sequence
$\{\nu_{\iota_k(i)}\}_{i=1}^{\infty}$ of finite
discrete Borel measures on $\rbb_+$ such that
$\{\at{\nu_{\xbf}}\}_{\xbf \in \xsc_{k+1}}$ are
pairwise disjoint countably infinite subsets of
$[\vartheta, \infty)$ and \eqref{ZenWielkiJest} holds
for all $\xbf \in \nbb^j$ and all $j \in \{0, \ldots,
k+1\}$. By induction on $k$, we then obtain a family
$\{\nu_{\xbf}\}_{\xbf \in \xsc}$ of finite discrete
Borel measures on $\rbb_+$ such that
$\{\at{\nu_{\xbf}}\}_{\xbf\in \xsc}$ are pairwise
disjoint countably infinite subsets of $[\vartheta,
\infty)$ and \eqref{ZenWielkiJest} holds for all $\xbf
\in \nbb^j$ and all $j\in \zbb_+$. Multiplying the
measures $\nu_{\xbf}$, $\xbf \in \xsc$, by appropriate
positive factors if necessary, we complete the proof.
   \end{proof}
From now on, we write $\zeta_{j_1, \ldots,j_k}$
instead of the more formal expression $\zeta_{(j_1,
\ldots,j_k)}$ whenever $(j_1, \ldots,j_k) \in \nbb^k$
and $k\Ge 2$.
   \begin{lem} \label{sq2}
If $n \in \nbb$ and $\vartheta \in [1,\infty)$, then
there exist a family $\{\varOmega_{\xbf}\}_{\xbf \in
\xsc}$ of countably infinite subsets of
$[\vartheta,\infty)$ and a discrete measure $\nu \in
\pscr{\vartheta}$ such that
   \begin{enumerate}
   \item[(i)] $\at{\nu} = \varOmega_0$,
   \item[(ii)] $\varOmega_0 = \bigsqcup_{j_1=1}^{\infty}
\varOmega_{j_1}$ and $\varOmega_{j_1, \ldots,
j_k} = \bigsqcup_{j_{k+1}=1}^{\infty}
\varOmega_{j_1, \ldots, j_k, j_{k+1}}$ for all
$(j_1, \ldots, j_k) \in \nbb^k$ and $k\in \nbb$,
   \item[(iii)] $\int_{\varOmega_{\xbf}} s^{k+n} \D
\nu(s) < \infty$ and $\int_{\varOmega_{\xbf}}
s^{k+n+1} \D \nu(s) = \infty$ for all $\xbf \in
\nbb^k$ and $k \in \zbb_+$.
   \end{enumerate}
   \end{lem}
   \begin{proof}
Applying Lemma \ref{sq1}, we get a family
$\{\nu_{\xbf}\}_{\xbf\in \xsc}$ of finite discrete
Borel measures on $\rbb_+$ satisfying the conditions
(i)-(iii) of this lemma. Define the set $\varOmega_0$~
by
   \begin{align} \label{bdw}
\varOmega_0 = \bigsqcup_{\xbf \in \xsc}
\varDelta_{\xbf} = \bigsqcup_{k=0}^\infty \,
\bigsqcup_{\xbf \in \nbb^k} \varDelta_{\xbf} \text{
with $\varDelta_{\xbf} = \at{\nu_{\xbf}}$ for every
$\xbf \in \xsc$}.
   \end{align}
It is plain that $\varOmega_0$ is a countably infinite
subset of $[\vartheta, \infty)$. Set
   \begin{align} \label{wazka}
\nu = \sum_{k=0}^\infty \, \sum_{\xbf \in \nbb^k}
\nu_{\xbf}.
   \end{align}
Clearly, $\nu$ is a discrete Borel measure on $\rbb_+$
satisfying (i). Since $\varOmega_0 \subseteq
[\vartheta, \infty) \subseteq [1,\infty)$, we infer
from Lemma \ref{sq1}(ii) that
   \begin{align} \notag
\nu(\rbb_+) \overset{(\mathrm{i})}\Le \int_0^\infty
s^n \D \nu(s) &\overset{\eqref{wazka}}=
\sum_{k=0}^\infty \, \sum_{\xbf \in \nbb^k}
\int_{\varDelta_{\xbf}} s^{n} \D \nu_{\xbf}(s)
   \\ \label{asa}
&\hspace{1ex}\Le \sum_{k=0}^\infty \, \sum_{\xbf \in
\nbb^k} \int_{\varDelta_{\xbf}} s^{k+n} \D
\nu_{\xbf}(s) \Le 2.
   \end{align}
This means that the measure $\nu$ is finite and
consequently, by (i), the closed support of $\nu$ is
contained in $[\vartheta,\infty)$. It follows from
Lemma \ref{sq1}(iii) that
   \begin{align*}
\int_{\varOmega_0} s^{n+1} \D\nu(s)
\overset{\eqref{wazka}} \Ge \int_{\varDelta_0} s^{n+1}
\D\nu_0(s) = \infty.
   \end{align*}
This and \eqref{asa} show that the inequality and the
equality in (iii) hold for $\xbf \in \nbb^0$ and
$k=0$.

Now we will construct a family
$\{\varOmega_{\xbf}\}_{\xbf \in \xsc \setminus
\nbb^0}$ of countably infinite subsets of
$[\vartheta,\infty)$ and a family $\{t_{\xbf}\}_{\xbf
\in \xsc \setminus \nbb^0} \subseteq
[\vartheta,\infty)$ which satisfy the following
conditions for all $l\in \nbb$ and $(j_1, \ldots, j_l)
\in \nbb^l$,
   \allowdisplaybreaks
   \begin{align}  \label{SukHyun1}
& \varOmega_0 = \bigsqcup_{j_1^\prime=1}^\infty
\varOmega_{j_1^\prime},
   \\  \label{SukHyun2}
&\text{if } l \Ge 2, \text{ then } \varOmega_{j_1,
\ldots, j_{l-1}} = \bigsqcup_{j_{l}^\prime=1}^{\infty}
\varOmega_{j_1, \ldots, j_{l-1}, j_{l}^\prime},
   \\  \label{SukHyun3}
& \text{$\{t_{j_1^\prime}\}_{j_1^\prime = 1}^\infty$
is an injective sequence in $\varDelta_0$ such that
$\varDelta_0 = \{t_{j_1^\prime} \colon j_1^\prime \in
\nbb\}$,}
   \\[1ex] \label{SukHyun4}
&\left. \hspace{-1ex}
   \begin{aligned}
\text{if } l\Ge 2, &\text{ then } \{t_{j_1, \ldots,
j_{l-1}, j_{l}^\prime}\}_{j_{l}^\prime=1}^\infty
\text{ is an injective sequence in $\bigsqcup_{\xbf
\in \xsc_{l-1}} \varDelta_{\xbf}$}
   \\
&\text{ such that }\{t_{j_1,\ldots, j_{l-1}}\} \sqcup
\varDelta_{j_1,\ldots, j_{l-1}} = \{t_{j_1,\ldots,
j_{l-1},j_l^\prime} \colon j_l^\prime \in \nbb\},
   \end{aligned}
   \hspace{1ex}\right\}
   \\[.5ex]   \label{SukHyun5}
& \varOmega_{j_1,\ldots, j_l} = \{t_{j_1,\ldots,
j_l}\} \sqcup \varDelta_{j_1,\ldots, j_l} \sqcup
\bigsqcup_{p=1}^\infty \, \bigsqcup_{(j_{l+1}^\prime,
\ldots, j_{l+p}^\prime) \in \nbb^{p}}^\infty
\varDelta_{j_1,\ldots, j_l, j_{l+1}^\prime, \ldots,
j_{l+p}^\prime}.
   \end{align}
Since $\xsc_k \subsetneq \xsc_{k+1}$ for every $k\in
\nbb$ and $\xsc = \bigcup_{k=1}^{\infty} \xsc_k $, we
can obtain the required families inductively by
constructing ascending sequences of families
$\{\varOmega_{\xbf}\}_{\xbf \in \xsc_k \setminus
\nbb^0}$ and $\{t_{\xbf}\}_{\xbf\in \xsc_k \setminus
\nbb^0}$ satisfying the conditions
\eqref{SukHyun1}-\eqref{SukHyun5} for all $l\in J_k$
and $(j_1, \ldots, j_l) \in \nbb^l$ (clearly, the
conditions \eqref{SukHyun2} and \eqref{SukHyun4} are
void for $l=1$).

For the base step ($k=1$), note that since
$\varDelta_0$ is a countably infinite subset of
$[\vartheta, \infty)$, there exists a sequence
$\{t_{j_1^\prime}\}_{j_1^\prime = 1}^\infty \subseteq
[\vartheta, \infty)$ which satisfies \eqref{SukHyun3}.
For $j_1 \in \nbb$, we define the set
$\varOmega_{j_1}$ by \eqref{SukHyun5} with $l=1$. It
follows from \eqref{bdw} and \eqref{SukHyun3} that
$\varOmega_{j_1}$, $j_1\in \nbb$, are well-defined
countably infinite subsets of $[\vartheta, \infty)$
that satisfy \eqref{SukHyun1}.

For the induction step, let $k$ be some unspecified
positive integer. Suppose we have constructed a family
$\{\varOmega_{\xbf}\}_{\xbf \in \xsc_k \setminus
\nbb^0}$ of countably infinite subsets of
$[\vartheta,\infty)$ and a family
$\{t_{\xbf}\}_{\xbf\in \xsc_k \setminus
\nbb^0}\subseteq [\vartheta, \infty)$ such that
\eqref{SukHyun1}-\eqref{SukHyun5} hold for all $l\in
J_k$ and $(j_1, \ldots, j_l) \in \nbb^l$. Let
$(j_1,\ldots, j_k) \in \nbb^{k}$. Since
$\varDelta_{j_1,\ldots, j_k}$ is a countably infinite
subset of $[\vartheta, \infty)$, we infer from
\eqref{SukHyun3} if $k=1$, or \eqref{SukHyun4} with
$l=k$ if $k \Ge 2$, that there exists a sequence
$\{t_{j_1, \ldots, j_k,
j_{k+1}^\prime}\}_{j_{k+1}^\prime = 1}^\infty
\subseteq [\vartheta, \infty)$ which satisfies
\eqref{SukHyun4} with $l=k+1$. For $j_{k+1} \in \nbb$,
we define the set $\varOmega_{j_1,\ldots, j_{k+1}}$ by
\eqref{SukHyun5} with $l=k+1$. It follows from
\eqref{SukHyun4} with $l=k+1$ that
$\varOmega_{j_1,\ldots, j_{k+1}}$, $j_{k+1} \in \nbb$,
are well-defined countably infinite subsets of
$[\vartheta,\infty)$ which satisfy \eqref{SukHyun2}
for $l=k+1$. This completes the induction step. Using
induction, we obtain the required systems
$\{t_{\xbf}\}_{\xbf \in \xsc \setminus \nbb^0}$ and
$\{\varOmega_{\xbf}\}_{\xbf \in \xsc \setminus
\nbb^0}$ satisfying \eqref{SukHyun1}-\eqref{SukHyun5}
for all $l\in \nbb$ and $(j_1, \ldots, j_l) \in
\nbb^l$.

Clearly, the so constructed family
$\{\varOmega_{\xbf}\}_{\xbf \in \xsc}$ satisfies (i)
and (ii). It remains to show that the inequality and
the equality in (iii) hold for all $\xbf \in \nbb^k$
and $k \in \nbb$. Using \eqref{SukHyun5} with $l=k$,
the conditions \eqref{bdw} and \eqref{wazka}, the fact
that $\varDelta_{\xbf} \subseteq [1, \infty)$ for
every $\xbf \in \xsc$ and Lemma \ref{sq1}(ii), we see
that for all $k\in \nbb$ and $(j_1, \ldots, j_k) \in
\nbb^k$,
   \allowdisplaybreaks
   \begin{align*}
\int_{\varOmega_{j_1, \ldots, j_k}} s^{k+n} \D \nu(s)
&\overset{\eqref{SukHyun5}}= \delta +
\sum_{p=1}^\infty \, \sum_{(j_{k+1}, \ldots, j_{k+p})
\in \nbb^{p}} \int_{\varDelta_{j_1, \ldots, j_{k+p}}}
s^{k+n} \D \nu_{j_1, \ldots, j_{k+p}}(s)
   \\
   & \hspace{1.4ex}\Le \delta + \sum_{p=1}^\infty \,
\sum_{(j_{k+1}, \ldots, j_{k+p}) \in \nbb^{p}}
\int_{\varDelta_{j_1, \ldots, j_{k+p}}} s^{k+p+n} \D
\nu_{j_1, \ldots, j_{k+p}}(s)
   \\
& \hspace{1.4ex} \Le \delta + \frac 12 < \infty,
   \end{align*}
where
   \begin{align*}
\delta=t_{j_1,\ldots,j_k}^{k+n}
\nu(\{t_{j_1,\ldots,j_k}\}) + \int_{\varDelta_{j_1,
\ldots, j_k}} s^{k+n} \D \nu_{j_1, \ldots, j_k}(s).
   \end{align*}
Arguing as above and using Lemma \ref{sq1}(iii), we
deduce that for all $k\in \nbb$ and $(j_1, \ldots,
j_k) \in \nbb^k$,
   \begin{align*}
\int_{\varOmega_{j_1,\ldots,j_k}} s^{k+n+1} \D \nu(s)
\Ge \int_{\varDelta_{j_1,\ldots,j_k}} s^{k+n+1} \D
\nu_{j_1,\ldots,j_k}(s) = \infty,
   \end{align*}
which yields (iii). Hence $\nu$ is a finite nonzero
discrete Borel measure on $\rbb_+$ satisfying (i) and
(iii). Replacing $\nu$ by $\nu(\rbb_+)^{-1}\nu$ if
necessary, we complete the proof.
   \end{proof}
   \begin{lem} \label{sq3}
Let $\tcal=(V,E)$ be an extremal directed tree.
Suppose $n\in \nbb$, $\vartheta \in [1,\infty)$ and
$w\in V$. Then there exist systems $\{\lambda_v\}_{v
\in \deso{w}}\subseteq (0,\infty)$ and $\{\mu_v\}_{v
\in \des{w}} \subseteq \pscr{\vartheta}$ such that for
every $u \in \des{w}$,
   \begin{align}  \label{consist7}
& \text{$\mu_u(\varDelta) = \sum_{v \in \dzi{u}}
\lambda_v^2 \int_\varDelta \frac{1}{s}\D \mu_v(s)$ for
every $\varDelta \in \borel{\rbb_+}$,}
   \\       \label{prop2}
&\int_0^\infty s^n \D \mu_u(s) < \infty \text{ and }
\int_0^\infty s^{n+1} \D \mu_u(s) = \infty.
   \end{align}
   \end{lem}
   \begin{proof}
Set
   \begin{align*}
E_{\xsc} = \big\{&(0,j_1)\colon j_1 \in \nbb \big\}
   \\
& \sqcup \bigsqcup_{k=1}^{\infty} \big\{\big((j_1,
\ldots, j_k), (j_1, \ldots, j_k, j_{k+1})\big)\colon
j_1, \ldots, j_k, j_{k+1} \in \nbb\big\}.
   \end{align*}
Note that $(\xsc,E_{\xsc})$ is a directed tree with
root $0$ (see Figure $1$).
   \begin{center}
   \includegraphics[width=8cm]
   {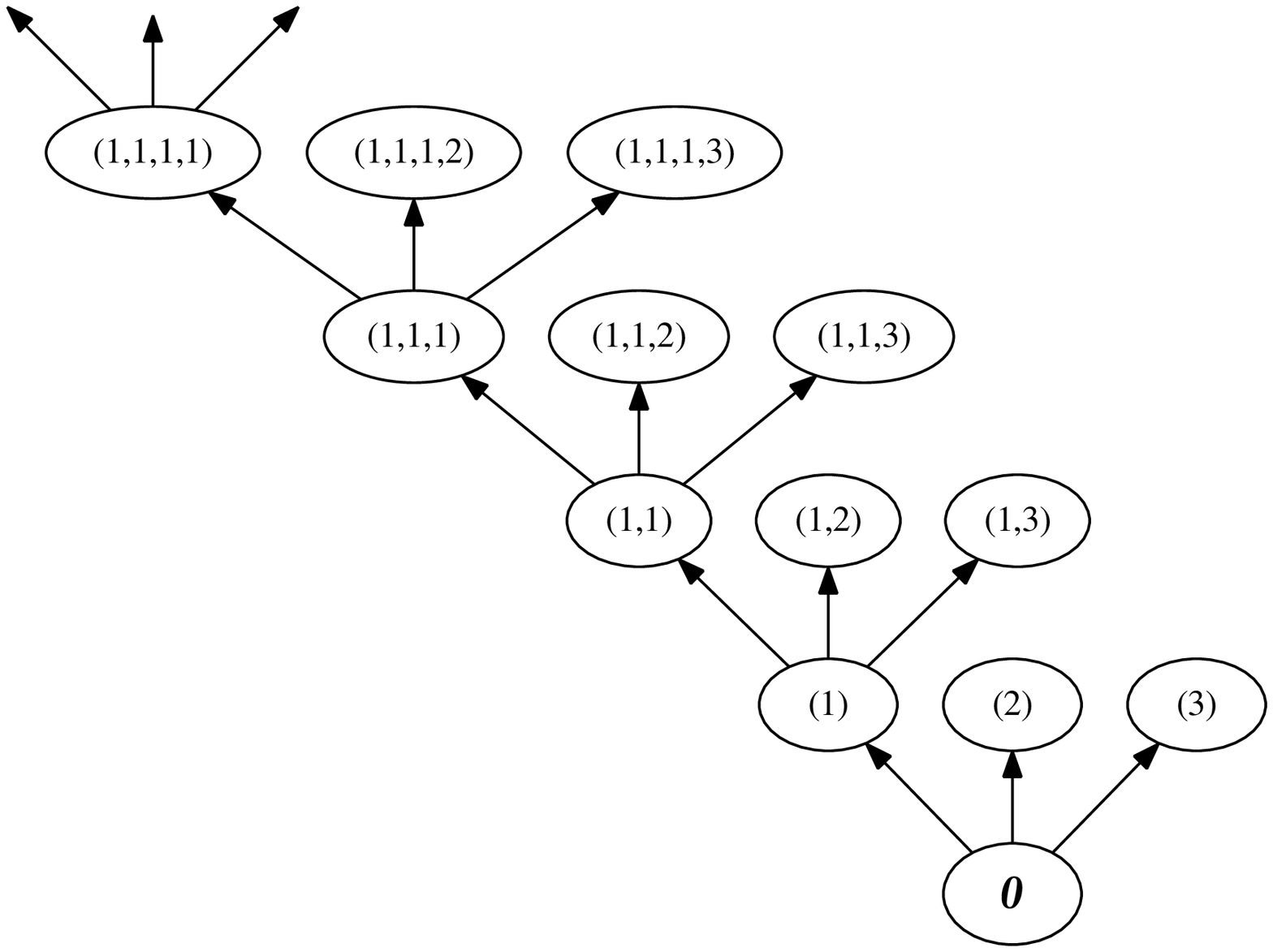}
   \\
   {\small {\sf Figure 1}}
   \end{center}
Using induction and the fact that $(\des{w}, E \cap
(\des{w}\times \des{w}))$ is an extremal directed tree
(because so is $\tcal$), we deduce that there exists a
family of distinct vertices $\{\xi_{\xbf}\}_{\xbf\in
\xsc}$ such that $\des{w}=\{\xi_{\xbf}\colon \xbf \in
\xsc\}$ and
   \begin{gather*}
\xi_{0} = w, \quad \dzi{\xi_{0}} =
\Big\{\xi_{j_1}\colon j_{1} \in \nbb\Big\},
   \\
\dzi{\xi_{j_1,\ldots,j_k}} =
\Big\{\xi_{j_1,\ldots,j_k,j_{k+1}}\colon j_{k+1} \in
\nbb\Big\}, \quad k\in \nbb, \, (j_1,\ldots,j_k) \in
\nbb^k.
   \end{gather*}
Then the mapping $\varPhi\colon \des{w} \to \xsc$
defined by
   \begin{align*}
\varPhi(\xi_{\xbf}) = \xbf, \quad \xbf \in \xsc,
   \end{align*}
is a graph isomorphism. In the rest of the proof we
will use this identification.

Let $\nu$ and $\{\varOmega_{\xbf}\}_{\xbf \in \xsc}$
be as in Lemma \ref{sq2} (with the same $n$ and
$\vartheta$). In view of Lemmata \ref{zespol} and
\ref{sq2}(iii), we have
   \begin{align} \label{poprdef}
0 < \int_{\varOmega_{\xbf}} s^k \D \nu(s) < \infty,
\quad \xbf \in \nbb^k, \, k\in \nbb.
   \end{align}
Set $\mu_{0}=\nu$. Then $\mu_{0} \in
\pscr{\vartheta}$. For a given $k \in \nbb$ and $(j_1,
\ldots, j_k) \in \nbb^k$, we define the Borel measure
$\mu_{j_1, \ldots, j_k}$ on $\rbb_+$ and
$\lambda_{j_1, \ldots, j_k} \in (0,\infty)$ by
   \allowdisplaybreaks
   \begin{align*}
\mu_{j_1, \ldots,j_k}(\varDelta) & =
\frac{\int_{\varDelta \cap \varOmega_{j_1,
\ldots,j_k}} s^k \D \nu(s)}{\int_{\varOmega_{j_1,
\ldots,j_k}} s^k \D \nu(s)}, \quad \varDelta \in
\borel{\rbb_+},
   \\[2ex]
\lambda_{j_1, \ldots, j_k} & =
   \begin{cases}
   \sqrt{\int_{\varOmega_{j_1, \ldots, j_k}} s^k
\D\nu(s)} & \text{ if } k=1,
   \\[2.5ex]
\sqrt{\frac{\int_{\varOmega_{j_1, \ldots, j_k}} s^k
\D\nu(s)}{\int_{\varOmega_{j_1, \ldots, j_{k-1}}}
s^{k-1} \D\nu(s)}} & \text{ if } k \Ge 2.
   \end{cases}
   \end{align*}
According to \eqref{poprdef}, $\mu_{j_1, \ldots, j_k}$
and $\lambda_{j_1, \ldots, j_k}$ are well-defined.
Since $\varOmega_{j_1, \ldots, j_k} \subseteq
[\vartheta,\infty)$, we see that $\mu_{j_1, \ldots,
j_k} \in \pscr{\vartheta}$.

Now, we verify that the conditions \eqref{consist7}
and \eqref{prop2} hold. Fix $k\in \nbb$ and $u=(j_1,
\ldots, j_k) \in \nbb^k$. Using \cite[Theorem
1.29]{Rud}, we infer from Lemma \ref{sq2}(ii) that
   \allowdisplaybreaks
   \begin{multline*}
\sum_{j_{k+1}=1}^\infty
\lambda_{j_1,\ldots,j_k,j_{k+1}}^2 \int_{\varDelta}
\frac{1}{s} \D \mu_{j_1,\ldots,j_k,j_{k+1}}(s) =
\sum_{j_{k+1}=1}^\infty \frac{\int_{\varDelta \cap
\varOmega_{j_1,\ldots,j_k,j_{k+1}}} s^k \D
\nu(s)}{\int_{\varOmega_{j_1, \ldots, j_{k}}} s^k
\D\nu(s)}
    \\
= \frac{\int_{\varDelta \cap \varOmega_{j_1, \ldots,
j_k}} s^k \D\nu(s)}{\int_{\varOmega_{j_1, \ldots,
j_{k}}} s^k \D\nu(s)} = \mu_{j_1, \ldots,
j_k}(\varDelta), \quad \varDelta \in \borel{\rbb_+},
   \end{multline*}
which gives \eqref{consist7}. By \cite[Theorem
1.29]{Rud} again and Lemma \ref{sq2}(iii), we have
   \begin{align*}
\int_0^\infty s^n \D \mu_{j_1,\ldots,j_k}(s) & =
\frac{\int_{\varOmega_{j_1,\ldots,j_k}} s^{k+n} \D
\nu(s)}{\int_{\varOmega_{j_1,\ldots,j_k}} s^k \D
\nu(s)} < \infty,
   \\
\int_0^\infty s^{n+1} \D \mu_{j_1,\ldots,j_k}(s) & =
\frac{\int_{\varOmega_{j_1,\ldots,j_k}} s^{k+n+1} \D
\nu(s)}{\int_{\varOmega_{j_1,\ldots,j_k}} s^k \D
\nu(s)} = \infty,
   \end{align*}
which means that \eqref{prop2} holds. Finally, arguing
as above and using Lemma \ref{sq2}(i), we can verify
that if $u=0$, then \eqref{consist7} and \eqref{prop2}
hold as well. This completes the proof.
   \end{proof}
   \begin{lem} \label{sq4}
Let $\tcal=(V,E)$ be an extremal directed tree, $w \in
V^{\circ}$, $x=\pa{w}$ and $n \in \nbb$. Suppose that
$\{\lambda_v\}_{v\in \deso{w}} \subseteq (0,\infty)$
and $\{\mu_v\}_{v \in \des{w}} \subseteq \pscr{1}$
satisfy \eqref{consist7} and \eqref{prop2} for every
$u \in \des{w}$. Then there exist $\{\lambda_v\}_{v\in
\deso{x} \setminus \deso{w}} \subseteq (0,\infty)$ and
$\{\mu_v\}_{v \in \des{x} \setminus \des{w}} \subseteq
\pscr{1}$ such that $\{\lambda_v\}_{v\in \deso{x}}$
and $\{\mu_v\}_{v \in \des{x}}$ satisfy
\eqref{consist7} and \eqref{prop2} for all $u \in
\des{x}$.
   \end{lem}
   \begin{proof}
By our assumption, there exists a sequence
$\{w_j\}_{j=0}^\infty$ of distinct vertices such that
$\dzi{x} = \{w_j\colon j\in \zbb_+\}$ and $w_0 = w$.
Note that
   \allowdisplaybreaks
   \begin{align} \label{des}
   \begin{aligned}
\deso{x}\setminus \deso{w} & = \{w\} \sqcup
\bigsqcup_{j=1}^\infty \des{w_j},
   \\
\des{x}\setminus \des{w} & = \{x\} \sqcup
\bigsqcup_{j=1}^\infty \des{w_j}.
   \end{aligned}
   \end{align}
Set $\vartheta_0=1$ and take a sequence
$\{\vartheta_j\}_{j=1}^\infty \subseteq [1,\infty)$
such that
   \begin{align} \label{dobor}
\sum_{j=1}^{\infty} \frac{1}{\vartheta_j} < \infty.
   \end{align}
Applying Lemma \ref{sq3}, we see that for each $j \in
\nbb$ there exist $\{\lambda_v\}_{v \in \deso{w_j}}
\subseteq (0,\infty)$ and $\{\mu_v\}_{v\in
\des{w_j}}\subseteq \pscr{\vartheta_j}$ which satisfy
\eqref{consist7} and \eqref{prop2} for all $u\in
\des{w_j}$. Define the sequence
$\{\tilde\lambda_{w_j}\}_{j=0}^\infty \subseteq
(0,\infty)$ by
   \begin{align*}
\tilde\lambda_{w_j} = \frac{1}{\sqrt{\vartheta_j
\int_0^\infty s^{n-1} \D\mu_{w_j}(s)}}, \quad j\in
\zbb_+.
   \end{align*}
By \eqref{prop2} and Lemma \ref{zespol}, the
quantities $\tilde\lambda_{w_j}$, $j \in \zbb_+$, are
well-defined. Noting that $\{\mu_{w_j}\}_{j=0}^\infty
\subseteq \pscr{1}$, we get
   \begin{align*}
\zeta:= \sum_{j=0}^\infty \tilde\lambda_{w_j}^2
\int_0^\infty \frac{1}{s} \D \mu_{w_j}(s) \Le
\sum_{j=0}^\infty \frac{1}{\vartheta_j \int_0^\infty
s^{n-1} \D\mu_{w_j}(s)} \Le \sum_{j=0}^\infty
\frac{1}{\vartheta_j} \overset{\eqref{dobor}} <
\infty,
   \end{align*}
and $\zeta > 0$. Set $\lambda_{w_j} = \tilde
\lambda_{w_j}/\sqrt{\zeta}$ for $j \in \zbb_+$ and
define the measure $\mu_x \in \pscr{1}$ by
   \begin{align*}
\mu_x(\varDelta) = \sum_{j=0}^\infty \lambda_{w_j}^2
\int_{\varDelta} \frac{1}{s} \D\mu_{w_j}(s), \quad
\varDelta \in \borel{\rbb_+}.
   \end{align*}
Clearly, with such $\{\lambda_v\}_{v\in \deso{x}
\setminus \deso{w}} \subseteq (0,\infty)$ and
$\{\mu_v\}_{v \in \des{x} \setminus \des{w}} \subseteq
\pscr{1}$ (cf.\ \eqref{des}), the systems
$\{\lambda_v\}_{v\in \deso{x}}$ and $\{\mu_v\}_{v \in
\des{x}}$ satisfy \eqref{consist7} for all $u \in
\des{x}$. It remains to prove that \eqref{prop2} holds
for $u=x$. For this, note that by \cite[Theorem
1.29]{Rud}, we have
   \begin{align*}
\int_0^\infty s^n \D\mu_x(s) = \frac{1}{\zeta}
\sum_{j=0}^\infty \tilde \lambda_{w_j}^2 \int_0^\infty
s^{n-1} \D\mu_{w_j}(s) = \frac{1}{\zeta}
\sum_{j=0}^\infty \frac{1}{\vartheta_j}
\overset{\eqref{dobor}} < \infty,
   \end{align*}
and
\allowdisplaybreaks
   \begin{align*}
\int_0^\infty s^{n+1} \D\mu_x(s) & = \frac{1}{\zeta}
\sum_{j=0}^\infty \tilde \lambda_{w_j}^2 \int_0^\infty
s^{n} \D\mu_{w_j}(s)
   \\
&= \frac{1}{\zeta} \sum_{j=0}^\infty
\frac{\int_0^\infty s s^{n-1}
\D\mu_{w_j}(s)}{\vartheta_j \int_0^\infty s^{n-1}
\D\mu_{w_j}(s)}
   \\
&\overset{(\star)}\Ge \frac{1}{\zeta}
\sum_{j=0}^\infty \frac{\vartheta_j \int_0^\infty
s^{n-1} \D\mu_{w_j}(s)}{\vartheta_j \int_0^\infty
s^{n-1} \D\mu_{w_j}(s)} = \infty,
   \end{align*}
where $(\star)$ follows from the fact that the closed
support of $\mu_{w_j}$ is contained in
$[\vartheta_j,\infty)$ for every $j\in \zbb_+$. This
completes the proof.
   \end{proof}
   \begin{rem}
Regarding the proof of Lemma \ref{sq4}, it is worth
pointing out that we can define the sequence
$\{\tilde\lambda_{w_j}\}_{j=0}^\infty \subseteq
(0,\infty)$ using a more general formula
   \begin{align*}
\tilde\lambda_{w_j} = \frac{1}{\sqrt{\delta_j
\int_0^\infty s^{n-1} \D\mu_{w_j}(s)}}, \quad j\in
\zbb_+,
   \end{align*}
where $\{\delta_j\}_{j=0}^\infty \subseteq (0,\infty)$
and $\{\vartheta_j\}_{j=0}^\infty \subseteq
[1,\infty)$ are such that
   \begin{align*}
\vartheta_0=1, \quad \sum_{j=0}^\infty
\frac{1}{\delta_j} < \infty \quad \text{and} \quad
\sum_{j=0}^\infty \frac{\vartheta_j}{\delta_j} =
\infty.
   \end{align*}
   \end{rem}
   \smallskip
   \begin{proof}[The final stage of the proof of Theorem \ref{main}]
If $\tcal$ has a root, then we can apply Lemma
\ref{sq3} (with $w=\koo$ and $\vartheta=1$) and then
Lemma \ref{ddn} and Theorem \ref{wsi}.

Now assume that the directed tree $\tcal$ is rootless.
Take $w_0 \in V$ and note that $V =
\bigcup_{j=0}^\infty \des{\paa^j(w_0)}$ (cf.\
\cite[Proposition 2.1.6]{j-j-s}). Applying induction
and Lemma \ref{sq4} successively to $w=\paa^j(w_0)$,
we get systems $\{\lambda_v\}_{v\in V} \subseteq
(0,\infty)$ and $\{\mu_v\}_{v\in V} \subseteq
\pscr{1}$ which satisfy \eqref{consist7} and
\eqref{prop2} for all $u\in V$. Finally, employing
Lemma \ref{ddn} and Theorem \ref{wsi} completes the
proof.
   \end{proof}
   \subsection*{Acknowledgement} A substantial part
of this paper was written while the first, the second
and the fourth author visited Kyungpook National
University during the autumns of 2013 and 2014. They
wish to thank the faculty and the administration of
this unit for their warm hospitality.
   
   \end{document}